\documentclass[oneside,english]{amsart}
\usepackage[T1]{fontenc}
\usepackage[latin9]{inputenc}
\usepackage{color}
\usepackage{mathrsfs}
\usepackage{amstext}
\usepackage{amsthm}
\usepackage{amssymb}
\usepackage{esint}

\makeatletter
\numberwithin{equation}{section}
\numberwithin{figure}{section}
\theoremstyle{plain}
\newtheorem{thm}{\protect\theoremname}
\theoremstyle{remark}
\newtheorem{rem}[thm]{\protect\remarkname}
\theoremstyle{plain}
\newtheorem{lem}[thm]{\protect\lemmaname}
\theoremstyle{plain}
\newtheorem{cor}[thm]{\protect\corollaryname}

\usepackage[small]{diagrams}

\makeatother

\usepackage{babel}
\providecommand{\corollaryname}{Corollary}
\providecommand{\lemmaname}{Lemma}
\providecommand{\remarkname}{Remark}
\providecommand{\theoremname}{Theorem}

\begin{document}
\title[Exactness of Lepage 2-forms and globally variational equations]{Exactness of Lepage 2-forms and globally variational differential
equations}
\author{Zbyn\v{e}k Urban \and  Jana Voln\'a}
\address{Z. Urban\\
Department of Mathematics, Faculty of Civil Engineering, V\v{S}B-Technical
University of Ostrava, Ludv\'ika Pod\'e\v{s}t\v{e} 1875/17, 708
33 Ostrava, Czech Republic}
\email{zbynek.urban@vsb.cz}
\address{J. Voln\'a\\
Department of Mathematics, Faculty of Civil Engineering, V\v{S}B-Technical
University of Ostrava, Ludv\'ika Pod\'e\v{s}t\v{e} 1875/17, 708
33 Ostrava, Czech Republic}
\email{jana.volna@vsb.cz}
\begin{abstract}
The exactness equation for Lepage $2$-forms, associated with variational
systems of ordinary differential equations on smooth manifolds, is
analyzed with the aim to construct a~concrete global variational
principle. It is shown that locally variational systems defined by
homogeneous functions of degree $c\neq0,1$ are automatically globally
variational. A~new constructive method of finding a~global Lagrangian
is described for these systems, which include for instance the geodesic
equations in Riemann and Finsler geometry.
\end{abstract}

\thanks{ZU appreciates support of the Visegrad grant No. 51810810 at the University
of Pre\v{s}ov.}
\maketitle

\section{Introduction}

Our aim is to study a~\emph{construction of a~global Lagrangian}
for globally variational equations on fibered tangent bundles of smooth
$m$-dimensional manifolds. In our recent work \cite{UrbanVolna-2manifolds},
we gave a~solution to this problem for $m=2$ by means of applying
the de Rham top-cohomology theory; in part we now generalize the methods
of \cite{UrbanVolna-2manifolds} to dimension $m\geq2$. In general,
however, the problem remains open since we proceeded in this paper
for a~class of \emph{homogeneous} differential equations of degree
$c\neq0,1$ only. The topic belongs to studies of the influence of
topology on variationality of differential equations, and on the existence
and a~construction of the corresponding local and global variational
principles; see Krupka, Urban and Voln\'a \cite{KUV}.

The existence of a~global variational principle for (ordinary or
partial) differential equations is influenced by the topology of the
underlying space. For \emph{ordinary} variational equations of arbitrary
order it depends on the \textit{second de Rham} cohomology group $H_{\mathrm{dR}}^{2}Y$
of the underlying fibered manifold $Y$: if $H_{\mathrm{dR}}^{2}Y$
is \emph{trivial}, then a~locally variational source form $\varepsilon$
on the $r$-th jet prolongation $J^{r}Y$ is also globally variational.
This important result is due to Takens \cite{Takens} (see also Krupka
\cite{Krupka-VarSeqMech}, and for further comments Krupka \emph{et
al.} \cite{KruMorUrbVol}), is, however, sheaf-theoretic and to the
authors' knowledge there is no general method how to construct a~global
Lagrangian for locally variational equations. Simple examples also
show that the well-known Vainberg-Tonti formula (cf. Tonti {[}19{]})
fails to produce global Lagrangians.

The main idea is based on solvability of the \emph{global} \emph{exactness}
equation for the~\emph{Lepage equivalent }$\alpha_{\varepsilon}$
of a~variational source form $\varepsilon$, associated with a~given
system of variational second-order ordinary differential equations
(cf. Krupka \cite{Krupka-Book}, Krupkov\'a and Prince \cite{Krupkova-Prince}).
Globally defined $2$-form $\alpha_{\varepsilon}$ on $\mathbb{R}\times TM$
represents an example of a~\emph{Lepage $2$-form} in Lagrangian
mechanics (see Krupkov\'{a} \cite{Krupkova1986}), satisfying the
equation $\alpha_{\varepsilon}=d\Theta_{\lambda}$, where $\Theta_{\lambda}$
is the well-known \emph{Cartan form}, which depends on the choice
of a~Lagrangian $\lambda$ whereas $d\Theta_{\lambda}$ \emph{does
not}. As a~result, we reduce the global exactness of the Lepage equivalent
$\alpha_{\varepsilon}$ of $\varepsilon$ to global exactness of a~certain
$2$-form globally defined on $M$ (cf. Theorems \ref{thm:K-local}
and \ref{thm:Global}). Apparently, the topology of $M$ decides on
global exactness of $\alpha_{\varepsilon}$. The meaning of Lepage
forms for the calculus of variations and their basic properties have
been reviewed by Krupka, Krupková and Saunders \cite{KKS}.

In the class of second-order ordinary differential equations, given
by \emph{variational} and \emph{homogeneous of degree} $c\neq0,1$
functions, we describe a~new \emph{construction} of a~global variational
principle, which does \emph{not} depend on the topology of the underlying
manifold. We prove that locally variational source forms with homogeneous
coefficients of degree $c\neq0,1$ are automatically globally variational,
and possess a~global Lagrangian given by Theorem \ref{thm:Main}.
Note that systems of second-order positive-homogenous (of degree 1)
differential equations, characterized by the well-known Zermelo conditions,
were also studied from the variational point of view by Urban and
Krupka \cite{UK-Zermelo,UK-HomogeneousEq}.

For notation and basic geometric structures well adapted to this work,
we refer to Krupka, Urban, and Voln\'a \cite{KUV}. Throughout, we
consider fibered manifolds which are the Cartesian products $Y=\mathbb{R}\times M$
over the \emph{real line} $\mathbb{R}$ and projection $\pi:\mathbb{R}\times M\rightarrow\mathbb{R}$,
where $M$ is a~general smooth manifold of dimension $\dim M=m\geq2$.
Clearly, the jet spaces $J^{1}Y$ and $J^{2}Y$ can be canonically
identified with products $\mathbb{R}\times TM$ and $\mathbb{R}\times T^{2}M$,
respectively, where $TM$ is the tangent bundle of $M$, and $T^{2}M$
denotes the manifold of \emph{second-order velocities} over $M$.
The canonical jet projections are denoted by $\pi^{2}:\mathbb{R}\times T^{2}M\rightarrow\mathbb{R}$,
$\pi^{2,0}:\mathbb{R}\times T^{2}M\rightarrow\mathbb{R}\times M$,
and $\pi^{2,1}:\mathbb{R}\times T^{2}M\rightarrow\mathbb{R}\times TM$.
Recall that elements of $T^{2}M$ are $2$-jets $J_{0}^{2}\zeta\in J^{2}\left(\mathbb{R},M\right)$
with origin $0\in\mathbb{R}$ and target $\zeta(0)\in M$. These jet
spaces are endowed with the natural fibered manifold structure: if
$\left(V,\psi\right)$, $\psi=\left(x^{i}\right)$, is a~chart on
$M$, then $\left(\mathbb{R}\times V,\mathrm{id}_{\mathbb{R}}\times\psi\right)$,
$\mathrm{id}_{\mathbb{R}}\times\psi=\left(t,x^{i}\right)$, is a~fibered
chart on $\mathbb{R}\times M$, and the associated charts on $\mathbb{R}\times TM$
and $\mathbb{R}\times T^{2}M$ reads $\left(\mathbb{R}\times V^{1},\mathrm{id}_{\mathbb{R}}\times\psi^{1}\right)$,
$\mathrm{id}_{\mathbb{R}}\times\psi^{1}=\left(t,x^{i},\dot{x}^{i}\right)$,
and $\left(\mathbb{R}\times V^{2},\mathrm{id}_{\mathbb{R}}\times\psi^{2}\right)$,
$\mathrm{id}_{\mathbb{R}}\times\psi^{2}=\left(t,x^{i},\dot{x}^{i},\ddot{x}^{i}\right)$,
respectively. Here $V^{1}$ and $V^{2}$ are preimages of $V$ in
the canonical tangent bundles projections $TM\rightarrow M$ and $T^{2}M\rightarrow M$.

The exterior algebra of differential forms on $\mathbb{R}\times TM$,
resp. $\mathbb{R}\times T^{2}M$, is denoted by $\Omega^{1}\left(\mathbb{R}\times M\right)$,
resp. $\Omega^{2}\left(\mathbb{R}\times M\right)$. By means of charts,
we put $hdt=dt$, $hdx^{i}=\dot{x}^{i}dt$, $hd\dot{x}^{i}=\ddot{x}^{i}dt$,
and for any function $f:\mathbb{R}\times TM\rightarrow\mathbb{R}$,
$hf=f\circ\pi^{2,1}$. These formulas define a~global homomorphism
of exterior algebras $h:\Omega^{1}\left(\mathbb{R}\times M\right)\rightarrow\Omega^{2}\left(\mathbb{R}\times M\right)$,
called the $\pi$-\emph{horizontalization}. A~$1$-form $\rho\in\Omega^{1}\left(\mathbb{R}\times M\right)$
is called \emph{contact}, if $h\rho=0$. With respect to a~chart
$\left(V,\psi\right)$, $\psi=\left(x^{i}\right)$, on $M$, every
contact $1$-form $\rho$ has an expression $\rho=A_{i}\omega^{i}$,
for some functions $A_{i}:\mathbb{R}\times V^{1}\rightarrow\mathbb{R},$
where $\omega^{i}=dx^{i}-\dot{x}^{i}dt$. For any differential $1$-form
$\rho\in\Omega_{1}^{1}\left(\mathbb{R}\times M\right)$, the pull-back
$\left(\pi^{2,1}\right)^{*}\rho$ has a~unique decomposition $\left(\pi^{2,1}\right)^{*}\rho=h\rho+p\rho$,
where $h\rho$, resp. $p\rho$, is $\pi^{2}$-\emph{horizontal} (respectively,
\emph{contact}) $1$-form on $\mathbb{R}\times T^{2}M$. This decomposition
can be directly generalized to arbitrary $k$-forms. For $k=2$, if
$\rho\in\Omega_{2}^{1}\left(\mathbb{R}\times M\right)$ is a~$2$-form
on $\mathbb{R}\times TM$, then we get $\left(\pi^{2,1}\right)^{*}\rho=p_{1}\rho+p_{2}\rho$,
where $p_{1}\rho$ (resp. $p_{2}\rho$) is the $1$-\emph{contact}
(respectively, $2$-\emph{contact}) component of $\rho$, spanned
by $\omega^{i}\wedge dt$, (respectively, $\omega^{i}\wedge\omega^{j}$).
Analogously, we employ these concepts on $\mathbb{R}\times T^{2}M$.

The results of this work can be generalized to higher-order variational
differential equations by means of similar methods. Another non-trivial
extension consists in replacing Cartesian product $\mathbb{R}\times M$
by a~general fibered manifold over $1$-dimensional base, and extension
to partial differential equations.

\section{Ordinary variational equations and Lepage 2-forms}

In coherence with the general theory of ordinary variational differential
equations on fibered spaces (cf. Krupkov\'a and Prince \cite{Krupkova-Prince},
and references therein), we give basic definitions and concepts, adapted
to our underlying structures.

Let $\varepsilon$ be a source form on $\mathbb{R}\times T^{2}M$,
i.e. $\pi^{2,0}$-horizontal $1$-contact $2$-form, locally expressed
as
\begin{equation}
\varepsilon=\varepsilon_{i}\omega^{i}\wedge dt,\label{eq:Source}
\end{equation}
with respect to a~chart $\left(V,\psi\right)$, $\psi=\left(x^{i}\right)$,
on $M$. In \eqref{eq:Source}, we suppose the coefficients $\varepsilon_{i}$,
$1\leq i\leq m$, are differentiable functions on $V^{2}\subset T^{2}M$,
and $\omega^{i}=dx^{i}-\dot{x}^{i}dt$, are contact $1$-forms on
$\mathbb{R}\times V^{1}$. These assumptions mean that we restrict
ourselves to autonomous systems of second-order differential equations,
defined by functions 
\begin{equation}
\varepsilon_{i}\left(x^{j},\dot{x}^{j},\ddot{x}^{j}\right)=0,\label{eq:System}
\end{equation}
for unknown differentiable curves $\zeta$ in $M$, $I\ni t\rightarrow\zeta(t)=\left(x^{j}(\zeta(t))\right)$
on an open interval $I\subset\mathbb{R}$.

Source form $\varepsilon$ \eqref{eq:Source} (or system \eqref{eq:System})
is called \emph{locally variational}, if there exists a~real-valued
function $\mathscr{L}:\mathbb{R}\times V^{2}\rightarrow\mathbb{R}$
such that system \eqref{eq:System} coincide with the \emph{Euler-Lagrange
equations} associated with $\mathscr{L}$, that is, $\varepsilon_{i}=E_{i}\left(\mathscr{L}\right)$
are the \emph{Euler-Lagrange expressions} of $\mathscr{L}$,
\[
E_{i}\left(\mathscr{L}\right)=\frac{\partial\mathscr{L}}{\partial x^{i}}-\frac{d}{dt}\frac{\partial\mathscr{L}}{\partial\dot{x}^{i}}+\frac{d^{2}}{dt^{2}}\frac{\partial\mathscr{L}}{\partial\ddot{x}^{i}}.
\]
$\mathscr{L}$ is called a~(local) \emph{Lagrange function} for $\varepsilon$.
By a~\emph{Lagrangian} for fibered manifold $\mathbb{R}\times M$
over $\mathbb{R}$ we call a~$\pi^{2}$-horizontal $1$-form $\lambda$
on an open subset $W^{2}\subset\mathbb{R}\times T^{2}M$; in a fibered
chart we have $\lambda=\mathscr{L}\left(t,x^{i},\dot{x}^{i},\ddot{x}^{i}\right)dt$.
The mapping $\lambda\rightarrow E_{\lambda}$, assigning to a~Lagrangian
$\lambda$ the \emph{Euler-Lagrange form} $E_{\lambda}$, is the well-known
\emph{Euler-Lagrange mapping} in the calculus of variations; in a~fibered
chart we have
\begin{equation}
E_{\lambda}=E_{i}\left(\mathscr{L}\right)\omega^{i}\wedge dt.\label{eq:EL-form}
\end{equation}
We remark that a~Lagrangian represents a~class of $1$-forms, and
a~source form represents a~class of $2$-forms in the (quotient)
\emph{variational sequence} over $W$ (see Krupka \cite{Krupka-VarSeqMech},
and also Krupka \emph{et al.} \cite{KUV-Miskolc}), in which the Euler-Lagrange
mapping is one of its morphisms.

The coefficients of a~locally variational source form $\varepsilon$
\eqref{eq:Source} coincide with the Euler-Lagrange expressions of
a~Lagrange function with respect to every fibered chart. Note also
that such a~Lagrange function can always be reduced to an equivalent~first-order
Lagrange function $\mathscr{L}=\mathscr{L}\left(t,x^{i},\dot{x}^{i}\right)$
for $\varepsilon$. Local Lagrange functions for $\varepsilon$, defined
on chart neighborhoods in $\mathbb{R}\times TM$, need \emph{not}
define a~global Lagrange function for $\varepsilon$ on $\mathbb{R}\times TM$.
If there exists a~Lagrange function $\mathscr{L}$ for $\varepsilon$,
defined on $\mathbb{R}\times TM$, we call $\varepsilon$ \emph{globally
variational}.

The following theorem summarizes necessary and sufficient conditions
for local variationality of $\varepsilon$.
\begin{thm}
\label{thm:LocalVariational}Let $\varepsilon$ be a source form on
$\mathbb{R}\times T^{2}M$, locally expressed by \eqref{eq:Source}
with respect to a~chart $\left(V,\psi\right)$, $\psi=\left(x^{i}\right)$,
on $M$. The following conditions are equivalent:

\emph{(a)} $\varepsilon$ is locally variational.

\emph{(b)} Functions $\varepsilon_{i}$, $1\leq i\leq m$, satisfy
the following system identically, 
\begin{align}
 & \frac{\partial\varepsilon_{i}}{\partial\ddot{x}^{j}}-\frac{\partial\varepsilon_{j}}{\partial\ddot{x}^{i}}=0,\nonumber \\
 & \frac{\partial\varepsilon_{i}}{\partial\dot{x}^{j}}+\frac{\partial\varepsilon_{j}}{\partial\dot{x}^{i}}-\frac{d}{dt}\left(\frac{\partial\varepsilon_{i}}{\partial\ddot{x}^{j}}+\frac{\partial\varepsilon_{j}}{\partial\ddot{x}^{i}}\right)=0,\label{eq:Helmholtz}\\
 & \frac{\partial\varepsilon_{i}}{\partial x^{j}}-\frac{\partial\varepsilon_{j}}{\partial x^{i}}-\frac{1}{2}\frac{d}{dt}\left(\frac{\partial\varepsilon_{i}}{\partial\dot{x}^{j}}-\frac{\partial\varepsilon_{j}}{\partial\dot{x}^{i}}\right)=0.\nonumber 
\end{align}

\emph{(c)} Functions $\varepsilon_{i}$, $1\leq i\leq m$, are linear
in the second derivatives, i.e.
\begin{align}
 & \varepsilon_{i}=A_{i}+B_{ij}\ddot{x}^{j},\label{eq:Epsilon-AB}
\end{align}
and the functions $A_{i},B_{ij}$, $1\leq i,j\leq m$, depend on $x^{i},\dot{x}^{i}$
only, and satisfy the conditions,
\begin{align}
 & B_{ij}=B_{ji},\quad\frac{\partial B_{ik}}{\partial\dot{x}^{j}}=\frac{\partial B_{jk}}{\partial\dot{x}^{i}},\label{eq:HelmholtzAB-1}\\
 & \frac{\partial A_{i}}{\partial\dot{x}^{j}}+\frac{\partial A_{j}}{\partial\dot{x}^{i}}-2\frac{\partial B_{ij}}{\partial x^{k}}\dot{x}^{k}=0,\label{eq:HelmholtzAB-2}\\
 & \frac{\partial A_{i}}{\partial x^{j}}-\frac{\partial A_{j}}{\partial x^{i}}-\frac{1}{2}\frac{\partial}{\partial x^{k}}\left(\frac{\partial A_{i}}{\partial\dot{x}^{j}}-\frac{\partial A_{j}}{\partial\dot{x}^{i}}\right)\dot{x}^{k}=0.\label{eq:HelmholtzAB-3}
\end{align}

\emph{(d)} The function
\begin{equation}
\mathscr{L}=\mathscr{L}_{T}-\frac{d}{dt}\left(x^{i}\int_{0}^{1}C_{i}\left(sx^{k},s\dot{x}^{k}\right)ds\right),\label{eq:Tonti-1stOrder}
\end{equation}
where functions $C_{i}$ are given by $B_{ij}=\partial C_{i}/\partial\dot{x}^{j}=\partial C_{j}/\partial\dot{x}^{i}$,
and
\begin{equation}
\mathscr{L}_{T}=x^{i}\int_{0}^{1}\varepsilon_{i}\left(sx^{k},s\dot{x}^{k},s\ddot{x}^{k}\right)ds,\label{eq:Tonti}
\end{equation}
 is a Lagrange function for $\varepsilon$ defined on $V^{1}$.

\emph{(e)} To every point of $\mathbb{R}\times T^{2}M$ there is a
neighborhood $W$ and a $2$-contact $2$-form $F_{W}$ on $W$ such
that the form $\alpha_{W}=\varepsilon|_{W}+F_{W}$ is closed. 

\emph{(f)} There exists a~closed $2$-form $\alpha_{\varepsilon}$
on $\mathbb{R}\times TM$ such that $\varepsilon=p_{1}\alpha_{\varepsilon}$.
If $\alpha_{\varepsilon}$ exists, it is unique and it has a chart
expression given by
\begin{equation}
\alpha_{\varepsilon}=\varepsilon_{i}\omega^{i}\wedge dt+\frac{1}{4}\left(\frac{\partial\varepsilon_{i}}{\partial\dot{x}^{j}}-\frac{\partial\varepsilon_{j}}{\partial\dot{x}^{i}}\right)\omega^{i}\wedge\omega^{j}+\frac{\partial\varepsilon_{i}}{\partial\ddot{x}^{j}}\omega^{i}\wedge\dot{\omega}^{j}.\label{eq:ALFA-Lepage}
\end{equation}
\end{thm}

The identities \eqref{eq:Helmholtz}, or equivalently \eqref{eq:HelmholtzAB-1}\textendash \eqref{eq:HelmholtzAB-3},
are called the \emph{Helmholtz conditions} of local variationality
(cf. Krupková and Prince \cite{Krupkova-Prince}, and references therein).
Formula \eqref{eq:Tonti} yields the \emph{Vainberg-Tonti} \emph{Lagrange
function} for a locally variational source form (see Tonti \cite{Tonti},
Krupka \cite{Krupka-Book}), which can always be reduced to first-order
Lagrange function \eqref{eq:Tonti-1stOrder}.
\begin{rem}
The Euler-Lagrange form $E_{\lambda}$ \eqref{eq:EL-form}, associated
with the Vainberg-Tonti Lagrangian $\lambda=\mathscr{L}dt$ \eqref{eq:Tonti},
coincides with source form $\varepsilon$, provided the Helmholtz
conditions \eqref{eq:Helmholtz} are satisfied. We also note that
Helmholtz conditions \eqref{eq:Helmholtz} yield the following identity

\begin{equation}
\frac{\partial B_{ik}}{\partial x^{j}}-\frac{\partial B_{jk}}{\partial x^{i}}-\frac{1}{2}\frac{\partial}{\partial\dot{x}^{k}}\left(\frac{\partial A_{i}}{\partial\dot{x}^{j}}-\frac{\partial A_{j}}{\partial\dot{x}^{i}}\right)=0,\label{eq:HelmholtzAB-Dependent}
\end{equation}
which is, however, \emph{dependent} on Helmholtz conditions on $A_{i}$,
$B_{ij}$, \eqref{eq:HelmholtzAB-1}\textendash \eqref{eq:HelmholtzAB-3}.
\end{rem}

A~$2$-form $\alpha$ on $\mathbb{R}\times TM$ is called a~\emph{Lepage
$2$-form}, if $\alpha$ is closed and admits a~decomposition $\left(\pi^{2,1}\right)^{*}\alpha=E+F$,
where $E$ is $\pi^{2,0}$-horizontal $1$-contact $2$-form on $\mathbb{R}\times T^{2}M$,
and $F$ is a~$2$-contact $2$-form on $\mathbb{R}\times T^{2}M$.
$2$-form $\alpha_{\varepsilon}$, described by Theorem \ref{thm:LocalVariational},
(f), \eqref{eq:ALFA-Lepage}, is a~Lepage $2$-form, called the \emph{Lepage
equivalent} of a~source form $\varepsilon$.

The notion of a~Lepage $2$-form in fibered mechanics is due to Krupkov\'a
\cite{Krupkova1986} (see also generalizations by Krupkov\'a and
Prince \cite{Krupkova-Prince-paper,Krupkova-Prince}), and it contributes
to the theory of Lepage forms introduced by Krupka \cite{Krupka-Book},
and references therein. For further application, we point out the
following result.
\begin{thm}
Every second-order Lagrangian $\lambda$ on $\mathbb{R}\times T^{2}M$
has a~unique Lepage equivalent $\Theta_{\lambda}$ on $\mathbb{R}\times T^{3}M$.
In a fibered chart on $\mathbb{R}\times M$, if $\lambda=\mathscr{L}dt$,
then $\Theta_{\lambda}$ has the expression
\begin{equation}
\Theta_{\lambda}=\mathscr{L}dt+\left(\frac{\partial\mathscr{L}}{\partial\dot{x}^{i}}-\frac{d}{dt}\frac{\partial\mathscr{L}}{\partial\ddot{x}^{i}}\right)\omega^{i}+\frac{\partial\mathscr{L}}{\partial\ddot{x}^{i}}\dot{\omega}^{i}.\label{eq:Cartan}
\end{equation}
Moreover, a~source form $\varepsilon$ on $\mathbb{R}\times T^{2}M$
is globally variational if and only if the equation
\[
\left(\pi^{2,1}\right)^{*}\alpha_{\varepsilon}=d\Theta_{\lambda}
\]
has a~global solution $\lambda$ on $\mathbb{R}\times T^{2}M$.
\end{thm}

\begin{rem}
$\Theta_{\lambda}$ \eqref{eq:Cartan} is the well-known \emph{Cartan
form} in Lagrangian mechanics; cf. Krupka, Krupkov\textcolor{black}{\'a}
and Saunders \cite{KKS}. Since $\Theta_{\lambda}$ is the Lepage
equivalent of Lagrangian $\lambda$, it satisfies $p_{1}d\Theta_{\lambda}=E_{\lambda}$.
We also point out $\Theta_{\lambda}$ depends on the choice of a~Lagrangian
$\lambda$ whereas $d\Theta_{\lambda}$ \emph{does not}; $d\Theta_{\lambda}$
is decomposable as $d\Theta_{\lambda}=E_{\lambda}+F$, where $E_{\lambda}$
\eqref{eq:EL-form} is the Euler-Lagrange form of $\lambda$, and
$F$ is a~$2$-contact $2$-form.
\end{rem}

\section{Exactness equation for Lepage 2-forms on $\mathbb{R}\times T^{2}M$}

Let $\varepsilon$ be a locally variational source form on $\mathbb{R}\times T^{2}M$,
and $\alpha_{\varepsilon}$ be the Lepage equivalent of $\varepsilon$
(Theorem \ref{thm:LocalVariational}, (f), \eqref{eq:ALFA-Lepage}).
Since $\alpha_{\varepsilon}$ is closed, it is also \emph{locally
exact} according to the Poincar\'e lemma. In this section, we study
the exactness equation for\emph{ }Lepage 2-form $\alpha_{\varepsilon}$
\emph{globally}, with the aim to construct a~\emph{global Lagrangian},
provided $\varepsilon$ is in addition globally variational. That
is, we search for a~solution $\mu$ on $\mathbb{R}\times TM$ of
the equation
\begin{equation}
\alpha_{\varepsilon}=d\mu.\label{eq:AlfaExactnessEq}
\end{equation}
Clearly, equation \eqref{eq:AlfaExactnessEq} need not have a~solution,
and even if solvability of \eqref{eq:AlfaExactnessEq} is assured,\emph{
no general} \emph{construction} of its solution is known.

Properties of the Cartan equivalent $\Theta_{\lambda}$ of a~global
Lagrangian $\lambda$ imply the following straightforward observation.
\begin{lem}
\label{lem:Lemma}Let $\alpha_{\varepsilon}$ be the Lepage equivalent
of a~globally variational source form $\varepsilon$. Suppose that
a~$1$-form $\mu$ on $\mathbb{R}\times TM$ is a~solution of \eqref{eq:AlfaExactnessEq}.
Then the horizontal component $h\mu$ of $\mu$ is a~Lagrangian on
$\mathbb{R}\times T^{2}M$ for $\varepsilon$.
\end{lem}

\begin{proof}
Since the Cartan equivalent~$\Theta_{\lambda}$ of a~global Lagrangian
$\lambda$ for $\varepsilon$ obeys the property $\alpha_{\varepsilon}=d\Theta_{\lambda}$,
we obtain $\mu=\Theta_{\lambda}+df$ for some function $f$ hence
$h\mu=\lambda+h\left(df\right)$. Thus, $h\mu$ and $\lambda$ are
equivalent Lagrangians for $\varepsilon$, whose Lagrange functions
differ by means of total derivative of $f$.
\end{proof}
The next lemma describes a~\emph{global} decomposition of $\alpha_{\varepsilon}$
into closed forms.
\begin{lem}
\label{lem:AlfaClosed}Let $\alpha_{\varepsilon}$ be the Lepage equivalent
of a~locally variational source form $\varepsilon$ on $\mathbb{R}\times T^{2}M$.
Then there is a~unique decomposition of $\alpha_{\varepsilon}$ on
$\mathbb{R}\times TM$,
\begin{equation}
\alpha_{\varepsilon}=\alpha_{0}\wedge dt+\alpha',\label{eq:ALFA--Decomp}
\end{equation}
where $\alpha_{0}$ and $\alpha'$ are closed forms defined on $TM$.
With respect to a~chart $\left(V,\psi\right)$, $\psi=\left(x^{i}\right)$,
on $M$, we have
\begin{align}
\alpha_{0} & =\left(A_{i}-\frac{1}{2}\left(\frac{\partial A_{i}}{\partial\dot{x}^{j}}-\frac{\partial A_{j}}{\partial\dot{x}^{i}}\right)\dot{x}^{j}\right)dx^{i}+B_{ij}\dot{x}^{j}d\dot{x}^{i},\label{eq:ALFA--0}
\end{align}
and
\begin{equation}
\alpha'=\frac{1}{4}\left(\frac{\partial A_{i}}{\partial\dot{x}^{j}}-\frac{\partial A_{j}}{\partial\dot{x}^{i}}\right)dx^{i}\wedge dx^{j}+B_{ij}dx^{i}\wedge d\dot{x}^{j}.\label{eq:ALFA--prime}
\end{equation}
\end{lem}

\begin{proof}
\textit{\emph{In every chart}} $\left(V,\psi\right)$, $\psi=\left(x^{i}\right)$,
on $M$, it is straightforward to verify that decomposition of $\alpha_{\varepsilon}$
\eqref{eq:ALFA--Decomp} holds for $\alpha_{0}$ and $\alpha'$, given
by formulas \eqref{eq:ALFA--0} and \eqref{eq:ALFA--prime}. Since
the Lepage equivalent $\alpha_{\varepsilon}$ of $\varepsilon$ is
closed, and $\alpha_{0}$, $\alpha'$ do not contain $dt$, it easily
follows that both $\alpha_{0}$, $\alpha'$ must be closed. To verify
that for instance $d\alpha'$ vanishes, we can also proceed directly
with the help of Helmholtz conditions \eqref{eq:HelmholtzAB-1}\textendash \eqref{eq:HelmholtzAB-3}
and \eqref{eq:HelmholtzAB-Dependent}.

It remains to show that the $2$-forms $\alpha_{0}$, $\alpha'$ are
(globally) defined on the tangent bundle $TM$. Since $\alpha_{\varepsilon}$
is globally defined on $\mathbb{R}\times TM$, it is sufficient to
show that $\alpha'$ is defined on $TM$. For an arbitrary coordinate
transformation $x^{i}=x^{i}\left(\bar{x}^{j}\right)$ on $M$, we
get the following identities,
\begin{align}
 & A_{i}=\bar{A}_{k}\frac{\partial\bar{x}^{k}}{\partial x^{i}}+\bar{B}_{kl}\frac{\partial\bar{x}^{k}}{\partial x^{i}}\frac{\partial^{2}\bar{x}^{l}}{\partial x^{p}x^{q}}\dot{x}^{p}\dot{x}^{q},\label{eq:Aux1}\\
 & B_{ij}=\bar{B}_{kl}\frac{\partial\bar{x}^{k}}{\partial x^{i}}\frac{\partial\bar{x}^{l}}{\partial x^{j}}.\label{eq:Aux2}
\end{align}
Differentiating \eqref{eq:Aux1} we obtain with the help of the Helmholtz
condition \eqref{eq:HelmholtzAB-1},
\begin{align}
 & \frac{\partial A_{i}}{\partial\dot{x}^{j}}-\frac{\partial A_{j}}{\partial\dot{x}^{i}}\label{eq:Aux3}\\
 & \;=\frac{\partial\bar{A}_{k}}{\partial\dot{\bar{x}}^{l}}\left(\frac{\partial\bar{x}^{l}}{\partial x^{j}}\frac{\partial\bar{x}^{k}}{\partial x^{i}}-\frac{\partial\bar{x}^{l}}{\partial x^{i}}\frac{\partial\bar{x}^{k}}{\partial x^{j}}\right)+2\bar{B}_{kl}\left(\frac{\partial\bar{x}^{k}}{\partial x^{i}}\frac{\partial^{2}\bar{x}^{l}}{\partial x^{j}x^{p}}-\frac{\partial\bar{x}^{k}}{\partial x^{j}}\frac{\partial^{2}\bar{x}^{l}}{\partial x^{i}x^{p}}\right)\dot{x}^{p}.\nonumber 
\end{align}
From \eqref{eq:Aux2}, \eqref{eq:Aux3}, and using the transformation
formulas
\begin{align*}
 & \frac{\partial x^{i}}{\partial\bar{x}^{p}}\frac{\partial\bar{x}^{p}}{\partial x^{j}}=\delta_{j}^{i},\quad\frac{\partial^{2}\bar{x}^{l}}{\partial x^{i}x^{j}}\frac{\partial x^{i}}{\partial\bar{x}^{p}}\frac{\partial x^{j}}{\partial\bar{x}^{q}}=-\frac{\partial\bar{x}^{l}}{\partial x^{j}}\frac{\partial^{2}x^{j}}{\partial\bar{x}^{p}\partial\bar{x}^{q}},
\end{align*}
we now obtain
\begin{align*}
 & \frac{1}{4}\left(\frac{\partial A_{i}}{\partial\dot{x}^{j}}-\frac{\partial A_{j}}{\partial\dot{x}^{i}}\right)dx^{i}\wedge dx^{j}+B_{ij}dx^{i}\wedge d\dot{x}^{j}\\
 & =\frac{1}{4}\left(\frac{\partial\bar{A}_{k}}{\partial\dot{\bar{x}}^{l}}-\frac{\partial\bar{A}_{l}}{\partial\dot{\bar{x}}^{k}}\right)d\bar{x}^{k}\wedge d\bar{x}^{l}+\bar{B}_{kl}\frac{\partial^{2}\bar{x}^{l}}{\partial x^{i}x^{j}}\frac{\partial x^{i}}{\partial\bar{x}^{p}}\frac{\partial x^{j}}{\partial\bar{x}^{q}}\dot{\bar{x}}^{q}d\bar{x}^{k}\wedge d\bar{x}^{p}\\
 & +\bar{B}_{kl}\frac{\partial\bar{x}^{l}}{\partial x^{j}}\frac{\partial^{2}x^{j}}{\partial\bar{x}^{p}\partial\bar{x}^{q}}\dot{\bar{x}}^{p}d\bar{x}^{k}\wedge d\bar{x}^{q}+\bar{B}_{kl}\frac{\partial\bar{x}^{l}}{\partial x^{j}}\frac{\partial x^{j}}{\partial\bar{x}^{q}}d\bar{x}^{k}\wedge d\dot{\bar{x}}^{q}\\
 & =\frac{1}{4}\left(\frac{\partial\bar{A}_{k}}{\partial\dot{\bar{x}}^{l}}-\frac{\partial\bar{A}_{l}}{\partial\dot{\bar{x}}^{k}}\right)d\bar{x}^{k}\wedge d\bar{x}^{l}+\bar{B}_{kl}d\bar{x}^{k}\wedge d\dot{\bar{x}}^{l},
\end{align*}
as required.
\end{proof}
\begin{lem}
\label{lem:Mu0}The equation 
\begin{equation}
\alpha_{0}\wedge dt=d\mu_{0}\label{eq:Alfa0-Exact}
\end{equation}
has always a solution $\mu_{0}=-t\alpha_{0}$ defined on $\mathbb{R}\times TM$.
With respect to a~chart $\left(V,\psi\right)$, $\psi=\left(x^{i}\right)$,
on $M$, $\mu_{0}$ is expressed as
\begin{align}
{\normalcolor {\color{red}}} & \mu_{0}=-\left(A_{i}-\frac{1}{2}\left(\frac{\partial A_{i}}{\partial\dot{x}^{j}}-\frac{\partial A_{j}}{\partial\dot{x}^{i}}\right)\dot{x}^{j}\right)tdx^{i}-B_{ij}\dot{x}^{i}td\dot{x}^{j},\label{eq:Mu0}
\end{align}
and the horizontal component $h\mu_{0}$ of $\mu_{0}$, defined on
$\mathbb{R}\times T^{2}M$, reads
\begin{equation}
h\mu_{0}=-\varepsilon_{i}\dot{x}^{i}tdt,\label{eq:H-Mu0}
\end{equation}
where $\varepsilon_{i}=A_{i}+B_{ij}\ddot{x}^{j}$, see \eqref{eq:Epsilon-AB}.
\end{lem}

\begin{proof}
From Lemma \ref{lem:AlfaClosed} it follows that the $2$-form $\mu_{0}=-t\alpha_{0}$
is globally defined on $\mathbb{R}\times TM$, and a~straightforward
calculation shows that $\mu_{0}$ solves equation \eqref{eq:Alfa0-Exact}.
The expression \eqref{eq:H-Mu0} can be then easily obtained in a~chart
by applying the horizontal morphism $h$ to the expression \eqref{eq:Mu0}.
\end{proof}
We now study the equation
\[
\alpha'=d\mu',
\]
where $\alpha'$ is given by formula \eqref{eq:ALFA--prime}. To this
purpose we define canonical \emph{local} sections and homotopy operators
as follows.

Let $\left(V,\psi\right)$, $\psi=\left(x^{i}\right)$, be a~fixed
chart on $M$, and $\left(V^{1},\psi^{1}\right)$, $\psi^{1}=\left(x^{i},\dot{x}^{i}\right)$,
be the associated chart on $TM$. We put for every $l$, $1\leq l\leq m$,
\begin{align}
 & \pi_{l}^{1}\left(x^{i},\dot{x}^{l},\dot{x}^{k}\right)=\left(x^{i},\dot{x}^{k}\right),\quad1\leq i\leq m,\quad l+1\leq k\leq m,\label{eq:Projekce}
\end{align}

and
\begin{align}
 & s_{l,\nu_{l}}^{1}\left(x^{i},\dot{x}^{k}\right)=\left(x^{i},\nu_{l},\dot{x}^{k}\right),\quad1\leq i\leq m,\quad l+1\leq k\leq m.\label{eq:Rezy}
\end{align}
Local projections $\pi_{l}^{1}$, defined by formula \eqref{eq:Projekce},
map open subsets $V_{l}^{1}$ of the chart domain $V^{1}$, described
by equations $\dot{x}^{j}=0$, $1\leq j\leq l-1$, onto $V_{l+1}^{1}$,
whereas local sections $s_{l,\nu_{l}}^{1}$ \eqref{eq:Rezy} of $\pi_{l}^{1}$
map $V_{l+1}^{1}$ into $V_{l}^{1}$ for every $l$, $1\leq l\leq m$.
Note that in this notation $V_{1}^{1}=V^{1}$, $V_{m+1}^{1}=V$ are
the chart domains in $TM$ and $M$, respectively.

Define local homotopy operators as follows. For every $l$, $1\leq l\leq m$,
let $K_{l}$ acts on (local) differential forms defined on $V_{l}^{1}\subset V^{1}$
by the formula
\begin{equation}
K_{l}\rho=\intop_{0}^{\dot{x}^{l}}\left(\pi_{l}^{1}\right)^{*}\left(s_{l,\nu_{l}}^{1}\right)^{*}\left(i_{\frac{\partial}{\partial\dot{x}^{l}}}\rho\right)d\nu^{l},\label{eq:K-homotopy}
\end{equation}
where $\pi_{l}^{1}$ and $s_{l,\nu_{l}}^{1}$ are given by \eqref{eq:Projekce}
and \eqref{eq:Rezy}, and the integration operation in \eqref{eq:K-homotopy}
is applied on\emph{ coefficients} of the corresponding differential
form.
\begin{thm}
\label{thm:K-local}Let $\alpha_{\varepsilon}$ be the Lepage equivalent
of $\varepsilon$, and let $\alpha'$ be the uniquelly given $2$-form
by means of the decomposition \eqref{eq:ALFA--Decomp}, with local
expression \eqref{eq:ALFA--prime}. Then
\begin{equation}
\alpha'-\omega=d\kappa,\label{eq:KEq}
\end{equation}
where
\begin{align}
\omega & =\left(s_{1,0}^{1}\circ s_{2,0}^{1}\circ\ldots\circ s_{m,0}^{1}\circ\pi_{m}^{1}\circ\pi_{m-1}^{1}\circ\ldots\circ\pi_{1}^{1}\right)^{*}\alpha'\label{eq:Omega}\\
 & =\frac{1}{4}\left(\frac{\partial A_{i}}{\partial\dot{x}^{j}}-\frac{\partial A_{j}}{\partial\dot{x}^{i}}\right)_{\left(x^{p},0\right)}dx^{i}\wedge dx^{j},\nonumber 
\end{align}
and
\begin{align}
\kappa & =\sum_{l=1}^{m}\left(\pi_{1}^{1}\right)^{*}\left(\pi_{2}^{1}\right)^{*}\ldots\left(\pi_{l-1}^{1}\right)^{*}K_{l}\left(\left(s_{l-1,0}^{1}\right)^{*}\ldots\left(s_{2,0}^{1}\right)^{*}\left(s_{1,0}^{1}\right)^{*}\alpha'\right)\label{eq:KappaFormula}\\
 & =-\sum_{l=1}^{m}\intop_{0}^{\dot{x}^{l}}B_{jl}\left(x^{p},0,\ldots,0,\nu_{(l)},\dot{x}^{l+1},\ldots,\dot{x}^{m}\right)d\nu_{(l)}\cdot dx^{j}.\nonumber 
\end{align}
\end{thm}

\begin{proof}
First, we prove the identity
\begin{equation}
\alpha'-\left(\pi_{1}^{1}\right)^{*}\left(s_{1,0}^{1}\right)^{*}\alpha'=d\left(K_{1}\alpha'\right).\label{eq:K1aux}
\end{equation}
Using the chart expression of $\alpha'$ \eqref{eq:ALFA--prime},
we get the left-hand side of \eqref{eq:K1aux} as
\begin{align*}
 & \alpha'-\left(\pi_{1}^{1}\right)^{*}\left(s_{1,0}^{1}\right)^{*}\alpha'\\
 & =\frac{1}{4}\left(\frac{\partial A_{i}}{\partial\dot{x}^{j}}-\frac{\partial A_{j}}{\partial\dot{x}^{i}}\right)dx^{i}\wedge dx^{j}+B_{ij}dx^{i}\wedge d\dot{x}^{j}\\
 & -\frac{1}{4}\left(\frac{\partial A_{i}}{\partial\dot{x}^{j}}-\frac{\partial A_{j}}{\partial\dot{x}^{i}}\right)_{\left(x^{p},0,\dot{x}^{2},\ldots,\dot{x}^{m}\right)}dx^{i}\wedge dx^{j}-\sum_{j=2}^{m}B_{ij}\left(x^{p},0,\dot{x}^{2},\ldots,\dot{x}^{m}\right)dx^{i}\wedge d\dot{x}^{j}.
\end{align*}
Using the definition $K_{l}$ \eqref{eq:K-homotopy}, we have $i_{\partial/\partial\dot{x}^{1}}\alpha'=-B_{1i}dx^{i}$
and the right-hand side of \eqref{eq:K1aux} reads
\begin{align*}
 & d\left(K_{1}\alpha'\right)\\
 & =d\left(\intop_{0}^{\dot{x}^{1}}\left(\pi_{1}^{1}\right)^{*}\left(s_{1,\nu}^{1}\right)^{*}\left(i_{\frac{\partial}{\partial\dot{x}^{1}}}\alpha'\right)d\nu\right)=-d\left(\intop_{0}^{\dot{x}^{1}}B_{1i}\left(x^{p},\nu,\dot{x}^{2},\ldots,\dot{x}^{m}\right)d\nu\right)\wedge dx^{i}\\
 & =\frac{1}{2}\left(\intop_{0}^{\dot{x}^{1}}\left(\frac{\partial B_{1i}}{\partial x^{j}}-\frac{\partial B_{1j}}{\partial x^{i}}\right)_{\left(x^{p},\nu,\dot{x}^{2},\ldots,\dot{x}^{m}\right)}d\nu\right)dx^{i}\wedge dx^{j}\\
 & \quad+B_{1i}dx^{i}\wedge d\dot{x}^{1}-\sum_{j=2}^{m}\left(\intop_{0}^{\dot{x}^{1}}\left(\frac{\partial B_{1i}}{\partial\dot{x}^{j}}\right)_{\left(x^{p},\nu,\dot{x}^{2},\ldots,\dot{x}^{m}\right)}d\nu\right)d\dot{x}^{j}\wedge dx^{i}.
\end{align*}
We now apply the Helmholtz conditions \eqref{eq:HelmholtzAB-1} and
\eqref{eq:HelmholtzAB-Dependent},
\[
B_{ij}=B_{ji},\quad\frac{\partial B_{ik}}{\partial\dot{x}^{j}}=\frac{\partial B_{jk}}{\partial\dot{x}^{i}},\quad\frac{\partial B_{ik}}{\partial x^{j}}-\frac{\partial B_{jk}}{\partial x^{i}}=\frac{1}{2}\frac{\partial}{\partial\dot{x}^{k}}\left(\frac{\partial A_{i}}{\partial\dot{x}^{j}}-\frac{\partial A_{j}}{\partial\dot{x}^{i}}\right),
\]
and obtain
\begin{align*}
 & d\left(K_{1}\alpha'\right)\\
 & =\frac{1}{4}\left(\intop_{0}^{\dot{x}^{1}}\frac{\partial}{\partial\dot{x}^{1}}\left(\frac{\partial A_{i}}{\partial\dot{x}^{j}}-\frac{\partial A_{j}}{\partial\dot{x}^{i}}\right)_{\left(x^{p},\nu,\dot{x}^{2},\ldots,\dot{x}^{m}\right)}d\nu\right)dx^{i}\wedge dx^{j}\\
 & \quad+B_{1i}dx^{i}\wedge d\dot{x}^{1}-\sum_{j=2}^{m}\left(\intop_{0}^{\dot{x}^{1}}\left(\frac{\partial B_{ij}}{\partial\dot{x}^{1}}\right)_{\left(x^{p},\nu,\dot{x}^{2},\ldots,\dot{x}^{m}\right)}d\nu\right)d\dot{x}^{j}\wedge dx^{i}\\
 & =\frac{1}{4}\left(\left(\frac{\partial A_{i}}{\partial\dot{x}^{j}}-\frac{\partial A_{j}}{\partial\dot{x}^{i}}\right)-\left(\frac{\partial A_{i}}{\partial\dot{x}^{j}}-\frac{\partial A_{j}}{\partial\dot{x}^{i}}\right)_{\left(x^{p},0,\dot{x}^{2},\ldots,\dot{x}^{m}\right)}\right)dx^{i}\wedge dx^{j}\\
 & \quad+B_{i1}dx^{i}\wedge d\dot{x}^{1}+\sum_{j=2}^{m}\left(B_{ij}-B_{ij}\left(x^{p},0,\dot{x}^{2},\ldots,\dot{x}^{m}\right)\right)dx^{i}\wedge d\dot{x}^{j},\\
 & =\frac{1}{4}\left(\frac{\partial A_{i}}{\partial\dot{x}^{j}}-\frac{\partial A_{j}}{\partial\dot{x}^{i}}\right)dx^{i}\wedge dx^{j}-\frac{1}{4}\left(\frac{\partial A_{i}}{\partial\dot{x}^{j}}-\frac{\partial A_{j}}{\partial\dot{x}^{i}}\right)_{\left(x^{p},0,\dot{x}^{2},\ldots,\dot{x}^{m}\right)}dx^{i}\wedge dx^{j}\\
 & \quad+B_{ij}dx^{i}\wedge d\dot{x}^{j}-\sum_{j=2}^{m}B_{ij}\left(x^{p},0,\dot{x}^{2},\ldots,\dot{x}^{m}\right)dx^{i}\wedge d\dot{x}^{j},
\end{align*}
as required to show \eqref{eq:K1aux}. By means of similar arguments
we observe that the following formula holds
\begin{align}
 & \left(s_{l-1,0}^{1}\right)^{*}\ldots\left(s_{1,0}^{1}\right)^{*}\alpha'\label{eq:Rekurentni}\\
 & \quad=\left(\pi_{l}^{1}\right)^{*}\left(s_{l,0}^{1}\right)^{*}\left(s_{l-1,0}^{1}\right)^{*}\ldots\left(s_{1,0}^{1}\right)^{*}\alpha'+d\left(K_{l}\left(\left(s_{l-1,0}^{1}\right)^{*}\ldots\left(s_{1,0}^{1}\right)^{*}\alpha'\right)\right)\nonumber 
\end{align}
for every $l$, $1\leq l\leq m$. Applying formula \eqref{eq:Rekurentni}
recurrently, we now easily obtain
\begin{align*}
 & \alpha'=\left(\pi_{1}^{1}\right)^{*}\left(\pi_{2}^{1}\right)^{*}\ldots\left(\pi_{m}^{1}\right)^{*}\left(s_{m,0}^{1}\right)^{*}\left(s_{m-1,0}^{1}\right)^{*}\ldots\left(s_{1,0}^{1}\right)^{*}\alpha'\\
 & +d\left(\sum_{l=1}^{m}\left(\pi_{1}^{1}\right)^{*}\left(\pi_{2}^{1}\right)^{*}\ldots\left(\pi_{l-1}^{1}\right)^{*}K_{l}\left(\left(s_{l-1,0}^{1}\right)^{*}\ldots\left(s_{2,0}^{1}\right)^{*}\left(s_{1,0}^{1}\right)^{*}\alpha'\right)\right),
\end{align*}
as required.
\end{proof}
The identity \eqref{eq:KEq} of Theorem \ref{thm:K-local} is formulated
by means charts. We now show that \eqref{eq:KEq} is a~global decomposition
of $\alpha'$.
\begin{thm}
\label{thm:Global}Both $\kappa$ \eqref{eq:KappaFormula} and $\omega$
\eqref{eq:Omega} define (global) differential $1$-forms on $TM$.
\begin{proof}
We prove that the local expressions for both $\omega$ \eqref{eq:Omega}
and $\kappa$ \eqref{eq:KappaFormula} coincide on the intersection
of two overlapping charts on $TM$. To this purpose let $\bar{\Psi}\circ\Psi^{-1}\left(x^{i},\dot{x}^{i}\right)=\left(\bar{x}^{j},\dot{\bar{x}}^{j}\right)$
be the coordinate transformation between charts $\left(V,\psi\right)$,
$\psi=(x^{i},\dot{x}^{i})$, and $\left(\bar{V},\bar{\psi}\right)$,
$\bar{\psi}=(\bar{x}^{i},\dot{\bar{x}}^{i})$, on $TM$, where $\bar{x}^{j}=\bar{x}^{j}\left(x^{i}\right)$
and $\dot{\bar{x}}^{j}=\dot{\bar{x}}^{j}\left(x^{i},\dot{x}^{i}\right)$.

1. From \eqref{eq:Aux3} we have
\[
\left(\frac{\partial A_{i}}{\partial\dot{x}^{j}}-\frac{\partial A_{j}}{\partial\dot{x}^{i}}\right)_{\left(x^{p},0\right)}=\left(\frac{\partial\bar{A}_{k}}{\partial\dot{\bar{x}}^{l}}-\frac{\partial\bar{A}_{l}}{\partial\dot{\bar{x}}^{k}}\right)_{\left(\bar{x}^{p},0\right)}\frac{\partial\bar{x}^{l}}{\partial x^{j}}\frac{\partial\bar{x}^{k}}{\partial x^{i}},
\]
hence the transformation of the local formula \eqref{eq:Omega} for
$\omega$ reads
\begin{align*}
 & \omega=\frac{1}{4}\left(\frac{\partial A_{i}}{\partial\dot{x}^{j}}-\frac{\partial A_{j}}{\partial\dot{x}^{i}}\right)_{\left(x^{p},0\right)}dx^{i}\wedge dx^{j}\\
 & =\frac{1}{4}\left(\frac{\partial\bar{A}_{k}}{\partial\dot{\bar{x}}^{l}}-\frac{\partial\bar{A}_{l}}{\partial\dot{\bar{x}}^{k}}\right)_{\left(\bar{x}^{p},0\right)}\frac{\partial\bar{x}^{l}}{\partial x^{j}}\frac{\partial\bar{x}^{k}}{\partial x^{i}}\frac{\partial x^{i}}{\partial\bar{x}^{u}}\frac{\partial x^{j}}{\partial\bar{x}^{v}}d\bar{x}^{u}\wedge d\bar{x}^{v}\\
 & =\frac{1}{4}\left(\frac{\partial\bar{A}_{k}}{\partial\dot{\bar{x}}^{l}}-\frac{\partial\bar{A}_{l}}{\partial\dot{\bar{x}}^{k}}\right)_{\left(\bar{x}^{p},0\right)}d\bar{x}^{k}\wedge d\bar{x}^{l},
\end{align*}
as required.

2. Consider the local expression for $\kappa$ \eqref{eq:KappaFormula}
to which we apply the change of variables theorem for integrals of
differential forms. Employing the corresponding transformation properties
described for every $l$, $1\leq l\leq m$, by
\begin{align*}
 & \bar{\Psi}\circ\Psi^{-1}\left(x^{i},(0,\ldots0,\nu_{(l)},\dot{x}^{l+1},\ldots,\dot{x}^{m})\right)=(\bar{x}^{j},\bar{\mu}_{(l)}^{j}),
\end{align*}
where
\[
\bar{\mu}_{(l)}^{q}=\frac{\partial\bar{x}^{q}}{\partial x^{l}}\nu_{(l)}+\sum_{k=l+1}^{m}\frac{\partial\bar{x}^{q}}{\partial x^{k}}\dot{x}^{k},\quad1\leq q\leq m,
\]
and for $1\leq s\leq l-1$, $l+1\leq k\leq m$,
\begin{align}
 & 0=\frac{\partial x^{s}}{\partial\bar{x}^{j}}\bar{\mu}_{(l)}^{j},\quad\nu_{(l)}=\frac{\partial x^{l}}{\partial\bar{x}^{j}}\bar{\mu}_{(l)}^{j},\quad\dot{x}^{k}=\frac{\partial x^{k}}{\partial\bar{x}^{j}}\bar{\mu}_{(l)}^{j},\label{eq:TransformationProperties}
\end{align}
we obtain using \eqref{eq:Aux2} the coordinate transformation for
$\kappa$ \eqref{eq:KappaFormula}, where the integrals over segments
are transformed into \emph{line} integrals,
\begin{align}
 & \kappa=-\sum_{l=1}^{m}\intop_{0}^{\dot{x}^{l}}B_{jl}\left(x^{p},0,\ldots,0,\nu_{(l)},\dot{x}^{l+1},\ldots,\dot{x}^{m}\right)d\nu_{(l)}\cdot dx^{j}\nonumber \\
 & =-\sum_{l=1}^{m}\intop_{\bar{\mu}_{(l)}^{p}=\sum_{k=l+1}^{m}\frac{\partial\bar{x}^{p}}{\partial x^{k}}\dot{x}^{k}}^{\bar{\mu}_{(l)}^{p}=\sum_{k=l}^{m}\frac{\partial\bar{x}^{p}}{\partial x^{k}}\dot{x}^{k}}\bar{B}_{uv}(\bar{x}^{i},\bar{\mu}_{(l)}^{i})\frac{\partial\bar{x}^{u}}{\partial x^{j}}\frac{\partial\bar{x}^{v}}{\partial x^{l}}\frac{\partial x^{j}}{\partial\bar{x}^{w}}\frac{\partial x^{l}}{\partial\bar{x}^{q}}d\bar{\mu}_{(l)}^{q}\cdot d\bar{x}^{w}\label{eq:KappaAux}\\
 & =-\sum_{l=1}^{m}\intop_{\bar{\mu}_{(l)}^{p}=\sum_{k=l+1}^{m}\frac{\partial\bar{x}^{p}}{\partial x^{k}}\dot{x}^{k}}^{\bar{\mu}_{(l)}^{p}=\sum_{k=l}^{m}\frac{\partial\bar{x}^{p}}{\partial x^{k}}\dot{x}^{k}}\bar{B}_{uv}(\bar{x}^{i},\bar{\mu}_{(l)}^{i})\frac{\partial\bar{x}^{v}}{\partial x^{l}}\frac{\partial x^{l}}{\partial\bar{x}^{q}}d\bar{\mu}_{(l)}^{q}\cdot d\bar{x}^{u}.\nonumber 
\end{align}
Since the coordinate functions $x^{i}$ and $\dot{x}^{k}$, $l+1\leq k\leq m$,
are constant with respect to the integration in \eqref{eq:KappaAux},
from \eqref{eq:TransformationProperties} we get
\[
\frac{\partial x^{j}}{\partial\bar{x}^{q}}d\bar{\mu}_{(l)}^{q}=0,\quad j\neq l,
\]
hence in \eqref{eq:KappaAux} for every $l$, $1\leq l\leq m$,
\begin{equation}
\frac{\partial\bar{x}^{v}}{\partial x^{l}}\frac{\partial x^{l}}{\partial\bar{x}^{q}}d\bar{\mu}_{(l)}^{q}=\sum_{j=1}^{m}\frac{\partial\bar{x}^{v}}{\partial x^{j}}\frac{\partial x^{j}}{\partial\bar{x}^{q}}d\bar{\mu}_{(l)}^{q}=\delta_{q}^{v}d\bar{\mu}_{(l)}^{q}=d\bar{\mu}_{(l)}^{v}.\label{eq:DifferentialProperties}
\end{equation}
Using \eqref{eq:DifferentialProperties}, formula \eqref{eq:KappaAux}
now reads
\begin{align}
 & \kappa=-\sum_{l=1}^{m}\intop_{\bar{\mu}_{(l)}^{p}=\sum_{k=l+1}^{m}\frac{\partial\bar{x}^{p}}{\partial x^{k}}\dot{x}^{k}}^{\bar{\mu}_{(l)}^{p}=\sum_{k=l}^{m}\frac{\partial\bar{x}^{p}}{\partial x^{k}}\dot{x}^{k}}\bar{B}_{uv}(\bar{x}^{i},\bar{\mu}_{(l)}^{i})d\bar{\mu}_{(l)}^{v}\cdot d\bar{x}^{u}\nonumber \\
 & =-\intop_{\bar{\mu}^{p}=0}^{\bar{\mu}^{p}=\dot{\bar{x}}^{p}}\bar{B}_{uv}(\bar{x}^{i},\bar{\mu}^{i})d\bar{\mu}^{v}\cdot d\bar{x}^{u}.\label{eq:LineIntegral}
\end{align}
Since the functions $\bar{B}_{uv}=\bar{B}_{uv}(\bar{x}^{i},\bar{\mu}^{i})$
satisfy the \emph{Helmholtz condition} \eqref{eq:HelmholtzAB-1},
\begin{align*}
 & \frac{\partial\bar{B}_{uw}}{\partial\mu^{v}}=\frac{\partial\bar{B}_{vw}}{\partial\mu^{u}},
\end{align*}
the line integrals \eqref{eq:LineIntegral} for every $u$, $1\leq u\leq m$,
are \emph{independent} upon choice of a~path connecting the points
$(0,0,\ldots,0)$ and $(\dot{\bar{x}}^{1},\dot{\bar{x}}^{2},\ldots,\dot{\bar{x}}^{m})$.
Thus, in \eqref{eq:LineIntegral} we are allowed to integrate over
segments on an $m$-dimensional cube, that is
\begin{align*}
 & \kappa=-\intop_{\bar{\mu}^{p}=0}^{\bar{\mu}^{p}=\dot{\bar{x}}^{p}}\bar{B}_{uv}(\bar{x}^{i},\bar{\mu}^{i})d\bar{\mu}^{v}\cdot d\bar{x}^{u}\\
 & =-\sum_{v=1}^{m}\intop_{0}^{\dot{\bar{x}}^{v}}\bar{B}_{uv}(\bar{x}^{i},0,\ldots,0,\bar{\mu}^{v},\dot{\bar{x}}^{v+1},\ldots,\dot{\bar{x}}^{m})d\bar{\mu}^{v}\cdot d\bar{x}^{u},
\end{align*}
proving that \eqref{eq:KappaFormula} defines global differential
$1$-form $\kappa$ on $TM$.
\end{proof}
\end{thm}

Theorem \ref{thm:Global} implies that $\alpha'=\omega+d\kappa$ is
globally defined. Using this fact and applying Lemma \ref{lem:AlfaClosed}
and \ref{lem:Mu0}, we get a~global decomposition of $\alpha_{\varepsilon}$,
\begin{equation}
\alpha_{\varepsilon}=\omega+d\left(\mu_{0}+\kappa\right).\label{eq:Alfa-Global-Decomp}
\end{equation}
Clearly, the problem of global exactness of the Lepage equivalent
$\alpha_{\varepsilon}$ is by means of \eqref{eq:Alfa-Global-Decomp}
reduced to global exactness of $2$-form $\omega$ \eqref{eq:Omega}
defined on $M$. In other words, if $\omega$ is globally exact, then
so is $\alpha_{\varepsilon}$ hence the source form $\varepsilon$
is globally variational.
\begin{rem}
In general, if $M$ is an $m$-dimensional smooth manifold and $\rho$
is a \emph{closed} differential $k$-form on $M$, $k\leq m$, then
the equation $\rho=d\mu$ need \emph{not} have a~(global) solution
$\mu$ on $M$. Indeed, it is the $k$-th de Rham cohomology group
$H_{\mathrm{dR}}^{k}M=\mathrm{Ker}\,d_{k}/\mathrm{Im}\,d_{k-1}$ which
decides about solvability of the exactness equation $\rho=d\mu$.
If $H_{\mathrm{dR}}^{k}M$ is \emph{trivial}, then $\rho=d\mu$ has
always a solution $\mu$ on $M$. Nevertheless, in this case ($H_{\mathrm{dR}}^{k}M=0$)
there is \emph{no general} constructive procedure of finding a~solution
$\eta$ for a~given closed $k$-form $\rho$, where $k<m$; if $k=m$,
to find a solution one can apply the \emph{top-cohomology} theorems
(cf. Lee \cite{Lee}).
\end{rem}

\begin{cor}
\label{cor:Simple}If the $2$-form $\omega$ \eqref{eq:Omega} vanishes,
i.e. the coefficients of $\omega$ satisfy
\begin{equation}
\left(\frac{\partial A_{i}}{\partial\dot{x}^{j}}-\frac{\partial A_{j}}{\partial\dot{x}^{i}}\right)_{\left(x^{p},0\right)}=0\label{eq:Simple}
\end{equation}
in every chart, then source form $\varepsilon$ is globally variational
and it admits a~Lagrangian on $\mathbb{R}\times TM$, namely
\[
\lambda=h\left(\mu_{0}+\kappa\right),
\]
where $\mu_{0}$ and $\kappa$ are given by \eqref{eq:Mu0} and \eqref{eq:KappaFormula},
respectively.
\end{cor}

\begin{proof}
This is an immediate consequence of Theorems \ref{thm:K-local} and
\ref{thm:Global}, and Lemmas \ref{lem:AlfaClosed} and \ref{lem:Mu0}.
Indeed, using \eqref{eq:Alfa-Global-Decomp} we get
\[
\alpha_{\varepsilon}=\alpha_{0}\wedge dt+\alpha'=d\left(\mu_{0}+\kappa\right),
\]
hence the horizontal part of $\mu_{0}+\kappa$ is a Lagrangian for
$\varepsilon$; cf. Lemma \ref{lem:Lemma}.
\end{proof}
\begin{rem}[$\dim M=2$]
 In our paper \cite{UrbanVolna-2manifolds}, we studied the exactness
equation for Lepage equivalents of source forms on $\mathbb{R}\times T^{2}M$,
where $M$ is $2$-dimensional connected smooth manifold. In the corresponding
decomposition \eqref{eq:Alfa-Global-Decomp}, the equation $d\eta=\omega$
is solvable and its solution $\eta$ can be constructed with the help
of the top-cohomology theorems. Examples of a~global Lagrangian construction
on concrete smooth $2$-manifolds (M\textcolor{black}{\"o}bius strip,
punctured torus) are also discussed.
\end{rem}

\section{Globally variational homogeneous equations of degree $c\protect\neq0,1$}

We briefly recall some basic facts on locally variational second-order
ordinary differential equations, given by homogeneous functions of
degree $c\neq0,1$. More detailed exposition with proofs can be found
in a~recent paper by Rossi \cite{Rossi Debrecen}. Our main result
consists in Theorem \ref{thm:Main}, showing that locally variational
source forms with homogeneous coefficients of degree $c\neq0,1$ are
automatically globally variational.

A~real-vauled function $F:T^{2}M\rightarrow\mathbb{R}$, resp. $F:TM\rightarrow\mathbb{R}$,
is called \emph{homogeneous of degree} $c$, if $F$ satisfies 
\begin{equation}
\frac{\partial F}{\partial\dot{x}^{i}}\dot{x}^{i}+2\frac{\partial F}{\partial\ddot{x}^{i}}\ddot{x}^{i}=cF,\quad\textrm{resp.}\,\,\frac{\partial F}{\partial\dot{x}^{i}}\dot{x}^{i}=cF,\label{eq:HomogeneityCondition}
\end{equation}
with respect to any chart on $M$. Note that for $c=1$, $F:TM\rightarrow\mathbb{R}$
satisfying the Euler condition \eqref{eq:HomogeneityCondition} is
called a~positive-homogenous function.

Let $\varepsilon$ be a~locally variational source form $\mathbb{R}\times T^{2}M$.
If the coefficients $\varepsilon_{i}=A_{i}+B_{ij}\ddot{x}^{j}$ of
$\varepsilon$ are homogeneous of degree $c\neq0,1$, then using \eqref{eq:HomogeneityCondition}
it is readily seen that $A_{i}$ are homogeneous of degree $c$, and
$B_{ij}$ are homogeneous of degree $c-2$, and vice versa. It is
also straightforward to show that locally variational $\varepsilon_{i}$
are homogeneous of degree $c$ if and only if $\varepsilon_{i}$ possess
a~homogeneous Lagrangian of degree $c$.

The following theorem characterizes the structure of locally variational
homogeneous source forms.
\begin{thm}
\label{thm:Homogeneous}Let $\varepsilon$ be a~source form on $\mathbb{R}\times T^{2}M$,
with coefficients homogeneous of degree $c\neq0,1$, and affine in
second derivatives, $\varepsilon_{i}=A_{i}+B_{ij}\ddot{x}^{j}$ \eqref{eq:Epsilon-AB}.
The following two conditions are equivalent:

\emph{(a)} $\varepsilon$ is locally variational,

\emph{(b)} functions\emph{ $A_{i}$}, $B_{ij}$, satisfy the subset
of Helmholtz conditions \eqref{eq:HelmholtzAB-1}-\eqref{eq:HelmholtzAB-2},
\begin{align*}
 & B_{ij}=B_{ji},\quad\frac{\partial B_{ik}}{\partial\dot{x}^{j}}=\frac{\partial B_{jk}}{\partial\dot{x}^{i}},\\
 & \frac{\partial A_{i}}{\partial\dot{x}^{j}}+\frac{\partial A_{j}}{\partial\dot{x}^{i}}-2\frac{\partial B_{ij}}{\partial x^{k}}\dot{x}^{k}=0.
\end{align*}
Moreover, if $\varepsilon$ is locally variational, then $A_{i}$
satisfy 
\begin{equation}
A_{i}=\frac{1}{c-1}\left(\frac{1}{2}\left(\frac{\partial B_{ij}}{\partial x^{k}}+\frac{\partial B_{ik}}{\partial x^{j}}\right)-\frac{1}{c}\frac{\partial B_{jk}}{\partial x^{i}}\right)\dot{x}^{j}\dot{x}^{k}.\label{eq:A-homogeneous}
\end{equation}
\end{thm}

\begin{proof}
See Rossi \cite{Rossi Debrecen}.
\end{proof}
Combining Theorem \ref{thm:Homogeneous} with the results of Section
3, summarized by Corollary \ref{cor:Simple}, we immediately obtain
the following consequence for variational and homogeneous of degree
$c\neq0,1$ equations.
\begin{thm}
\label{thm:Main}Let $\varepsilon$ be a~locally variational source
form on $\mathbb{R}\times T^{2}M$, with coefficients homogeneous
of degree $c\neq0,1$. Then $\varepsilon$ is also globally variational,
and it admits a~global Lagrangian given by $\lambda=h\left(\mu_{0}+\kappa\right)$,
where $\mu_{0}$ and $\kappa$ are $1$-forms on $\mathbb{R}\times TM$
given by \eqref{eq:Mu0} and \eqref{eq:KappaFormula}, respectively.
\end{thm}

\begin{proof}
From the assumptions on $\varepsilon_{i}=A_{i}+B_{ij}\ddot{x}^{j}$
it follows that $A_{i}$ has the expression given by Theorem \ref{thm:Homogeneous},
\eqref{eq:A-homogeneous}. Hence condition \eqref{eq:Simple} holds,
and by Corollary \ref{cor:Simple} $\varepsilon$ is globally variational,
possessing a~global Lagrangian $\lambda=h\left(\mu_{0}+\kappa\right)=\mathscr{L}dt$,
where
\[
\mathscr{L}=-\varepsilon_{i}\dot{x}^{i}t-\sum_{l=1}^{m}\intop_{0}^{\dot{x}^{l}}B_{jl}\left(x^{p},0,\ldots,0,\nu_{(l)},\dot{x}^{l+1},\ldots,\dot{x}^{m}\right)d\nu_{(l)}\dot{x}^{j}.
\]
\end{proof}
\begin{rem}
Standard examples of variational and homogeneous of degree $2$ equations
are, for instance, the geodesic equations in Riemann geometry, geodesic
equations of a~spray in Finsler geometry, as well as the geodesic
equations of a~Cartan connection (or metrizable connection, cf. Krupka and Sattarov \cite{KrupkaSattarov}) on a~tangent bundle, associated with a~Finsler structure. Here, the functions $-B_{ij}$ are equal to components of metrics (Riemannian or Finsler), i.e. $B_{ij}=- g_{ij}$. Although it
is straightforward from the nature of these equations arising from the (global)
energy Lagrangian,
we point out that global variationality of these systems follows from
Theorem \ref{thm:Main}.
\end{rem}

\end{document}